\documentclass{article}
\usepackage{amssymb,amsfonts,amsmath,amsthm,amscd}
\usepackage{makeidx}
\usepackage[english]{babel}
\makeindex
\usepackage[x11names]{xcolor}
\usepackage{tikz-cd}  

\usepackage{hyperref}
\input xy
\xyoption{all}

\usepackage{fancyhdr}
\usepackage{makeidx}
\makeindex 
\usepackage{vmargin} 
\setmargins{3.3cm}       
{3cm}                        
{14.5cm}                      
{23cm}                    
{10pt}                           
{0.5cm}                           
{0pt}                             
{1cm}                           

\numberwithin{equation}{section}

\newtheorem{teo}{Theorem}[section]
\newtheorem{pro}[teo]{Proposition}
\theoremstyle{definition}
\newtheorem{defi}[teo]{Definition}

\newtheorem{lem}[teo]{Lemma}
\newtheorem{ejem}[teo]{Example}
\newtheorem{Notac}[teo]{Notation}
\newtheorem{rem}[teo]{Remark}
\newtheorem{coro}[teo]{Corollary}

\newcommand{\m}{{}^{-1}}
\newcommand{\af}{\alpha}

\newcommand{\ep}{\epsilon}
\newcommand{\E}{\mathcal{E}}
\newcommand{\Z}{\mathbb{Z}}
\newcommand{\K}{\mathbb{K}}
\newcommand{\C}{\mathcal{C}}

\newcommand{\ot}{\otimes}
\newcommand{\ind}{{\rm Ind}}
\newcommand{\coind}{{\rm Coind}}
\newcommand{\pa}{\partial}

\title{\textbf{On the structure of nearly epsilon and epsilon-strongly graded rings}}
\author{
  Luis Mart\'inez, H\'{e}ctor Pinedo and  Yerly Soler\\
   \small  Escuela de Matem\'{a}ticas\\
   \small Universidad Industrial de Santander\\
   \small Cra. 27 calle 9, Bucaramanga, Colombia\\
   \small  e-mail: luchomartinez9816@hotmail.com , hpinedot@uis.edu.co, yervane03@gmail.com\\
   }
   \date{\today}

\begin{document}

\maketitle
\begin{abstract}
In this work  we study  the classes of epsilon and nearly epsilon-strongly graded rings by a group $G$. In particular, we extend Dade's theorem to the realm of nearly epsilon-strongly graded rings.  Moreover, we introduce the category  SIM$S-$gr  of symmetrically graded modules and use it to present a new characterization of  strongly graded rings.  A functorial approach is used to obtain  a characterization of  epsilon-strongly graded rings.  Finally, we determine conditions for which an epsilon-strongly graded ring can be written as a direct sum of strongly graded rings and a trivially graded ring. 
\end{abstract}

\section{Introduction} Let $G$ be a group with identity element $e,$ and  $S$ be an associative ring equipped with a non-zero multiplicative identity $1$.
 We say that $S$ is  {\it graded} by $G$ if  for each $g \in G$, there 
is an additive subgroup $S_g$ of $S$
such that $S = \bigoplus\limits_{g \in G} S_g$ and
$S_g S_h \subseteq S_{gh},$ where for  $X$ and $Y$  nonempty subsets of $S$ one writes  $XY$ for the set of  all finite sums of elements of the form $xy$ with $x \in X$ and $y \in Y.$  The ring $S$ is {\it strongly graded} if $S_gS_h=S_{gh},$ for all $g,h\in G.$ 
The class of group graded rings contains several important mathematical structures, such as
polynomial rings, skew and twisted group rings,  matrix rings, crossed products and partial versions of these (see  for instance  \cite{D} and \cite{nas04}). Therefore, a theory of group graded rings can be applied to the study of completely different types of constructions, this serves 
as unification of a multitude of known theorems concerning these. In particular, the foundation of the general theory of strongly graded rings was presented  by Dade in his fundamental paper \cite{D}. Since then, many structural properties of strongly graded rings have been established (see e.g. \cite{nas04}).

On the other hand, partial actions of groups have been introduced in the theory of operator algebras giving powerful tools of their study, and in  a pure algebraic context were first studied in \cite{DE}, later  the possibility to construct a crossed product based on a partial action suggested the idea of creating a corresponding Galois Theory \cite{DFP}.  Crossed products related to partial actions are graded rings which are not necessarily strong, but belong to a more general class, the so called {\it epsilon-strongly graded rings}  (see Definition \ref{definitionepsilon}), this class was  recently  introduced in \cite{NYOP} and has being a subject of increasingly study (see \cite{L, L1, L2,  NYO, NYOP, NyOP3}).  Relevant families of rings which can be endowed with an  epsilon-strong gradation include  Morita rings, Leavitt Path Algebras associated to finite graphs, crossed product by partial actions  and corner skew polynomial rings. An important  class in the study of epsilon-strongly graded ring is the class of  {\it nearly epsilon-strongly graded rings}, which was  introduced in \cite{NYO} with the purpose to  study Leavitt path algebras associated  to arbitrary graphs. The classes of epsilon and nearly epsilon-strongly graded rings  will be our  principal object of study.

This work is structured as follows. After the introduction we give in Chapter \ref{pre} some preliminary notions that will be helpful in the work, later  in Chapter \ref{cate} we treat epsilon and   nearly epsilon-strongly graded rings from a categorical point of view, that is given a  graded ring $S$ with principal component $R,$ we present in Theorem \ref{teo1.1} a generalization of the classical Dade's theorem for strongly graded rings to the frame of nearly  epsilon-strongly graded rings. Further, Proposition \ref{indcoind} gives us some new characterization of the fact that $S$ is epsilon-strong by relating the categories $S$\text{-}gr  and  $R$\text{-}mod,  and we use this to give some properties preserved by the functors $\ind$ and $\coind.$   
In Chapter \ref{epfort} we treat the problem of  finding strongly-graded rings inside  an epsilon-strongly  graded ring (see Proposition \ref{pro2.2} and Theorem \ref{teo2.1}), this problem is related to the question of finding  global Galois extension inside partial Galois extension \cite{KuoSz, KuoSzII}.  Examples with Leavitt path algebras and matrix rings are considered.  

{\bf Conventions:} For functors $F$ and $F'$ we write $F\simeq F'$ to indicate that they are naturally equivalent.  Moreover, in this work  $S=\bigoplus_{g\in G}S_g$ will denote an associative unital  ring graded by a group $G$ with principal component $R=S_e,$   and unadorned $\otimes$ means $\otimes_R.$

\section{Preliminaries}\label{pre}
For the reader's convenience, we record from  \cite{nas04} some notions that will be used throughout this work. Let $S$ be  a  $G$-graded ring,   we say that a  left
$S$-module $M$ is {\it graded} if there is a family of additive
subgroups, $M_{g}$, $g\in G$, of $M$ such that $M
= \bigoplus\limits_{g \in G} M_{g}$, and $S_{g} M_{h} \subseteq 
M_{gh}$ for all $g,h\in
G.$ Any element $m\in M$ has a unique decomposition $m=\sum_{g\in G}m_g$ with $m_g\in M_g$ and all but a finite number of the $m_g$ are nonzero, the nonzero elements $m_g$ in the decomposition of $m$ are called the  {\it homogeneous components of m} and we write $\partial(m_g)=g.$ A submodule $N$ of the graded $S$-module $M$ is 
 said to be graded if  the equality $N=\bigoplus_{g\in G}(N\cap M_g)$ holds. If \{0\} and $M$ are the only graded submodules  of $M,$ then $M$ is said to be $gr$-simple.

  Let $R$-$mod$ and $S$-$gr$ denote  the categories of left $R$-modules and graded $S$-modules, respectively. The morphisms in  $S$-$gr$ are  $S$-linear maps $f : M \rightarrow
M'$ with the property $f(M_{g}) \subseteq M'_{g}$,
$g\in G$. It is known that $R$-$mod$ and $S$-$gr$ are abelian categories with enough projectives. There are several functors relating the categories  $R$-$mod$ and $S$-$gr$, among of them we consider      $\ind,\coind: R\text{-}mod\rightarrow S\text{-}gr$ and $(-)_g: S\text{-}gr \to R\text{-}mod, g\in G,$ which are defined as follows: the functor $(-)_g$ is the projection onto the $g$-th component, and  given a left $R$-module $M$  and $g\in G$ we set  $\ind(M):=S\otimes M$ with gradation $(S\otimes M)_g=S_g\otimes M,$ and $\coind(M)$ is the graded module with $g$-th component   $\coind(M)_g:=\{\varphi\in hom_{R}(S,M): \varphi(S_t)=\{0\},\forall  t\in G\setminus \{g^{-1}\}\}$.  Moreover,  for a morphism  $f:M\to M'$ in $R$-$mod$  we set $\ind(f):={\rm id}_S\otimes f$ and $\coind(f):=\coind(M)\ni \varphi\mapsto f\circ \varphi \in  \coind(M')$.  For a fixed  $g\in G$  the $g$-{\it suspension} of $M$ is the graded module $M(g)$ which coincides with $M$ as sets but has gradation $M(g)_h=M_{hg}, h\in G,$ this induces a {\it suspension functor} $T_g: S\text{-}gr\to  S\text{-}gr, $ sending $M$ to $M(g).$ Now consider the class
$\C_g=\{M\in S\text{-}gr \mid M_g=\{0\} \}$ and for $M\in S\text{-}gr $ we define  $t_{\C_g}(M)$ as the sum of the graded  submodules  of $M$ belonging to  $\C_g.$ If $t_{\C_g}(M)=\{0\}$ then $M$ is called {\it $\C_g$-torsion free}. 

For a deep study of categorical methods in graded rings, the interested reader may consult \cite{AR}.

\section{Categorical  and ring theoretical aspects relating    nearly  epsilon and epsilon-strongly graded rings}\label{cate}

\subsection{A Dade like  theorem  for  nearly epsilon-strongly graded rings}
In \cite{L1} the author dealt with several classes of graded rings (see \cite[Definition 3.3]{L1}) between them the most general are the so-called nearly epsilon-strongly graded rings. Which we recall in the next.
\begin{defi}
The ring  $S$ is called  nearly  epsilon-strongly graded if  for each  $g\in G$ the additive group $S_g$ is a $s$-unital  $(S_gS_{g^{-1}},S_{g^{-1}}S_g)$-bimodule.
\end{defi}
Important examples of nearly epsilon-strongly graded rings are induced quotient group gradings of epsilon-strongly graded \cite[Proposition 5.8]{L1} and Leavitt path algebras  (even for infinite graphs) with coefficients in unital rings endowed  with any standard $G$-grading  \cite[Theorem 4.2]{NYO}.

Following \cite[Definition 4.5]{CEP2016} we say that $S$ is \emph{symmetrically graded} if for every $g \in G$, the equality $S_g S_{g^{-1}} S_g = S_g$ holds.

For the reader's convenience we recall the following.
\begin{pro}\cite[Proposition 3.3]{NYO}\label{pro1.5}
The following assertions are equivalent:\begin{enumerate}
    \item [$(i)$] The ring S is nearly epsilon-strongly G-graded;
    \item [$(ii)$] $S$ is symmetrically  graded  and for each $g\in G$ the ring $S_gS_{g^{-1}}$ is $s$-unital;
    \item [$(iii)$] For each $g\in G$ and every $s\in S_g,$ there are  $\epsilon_g(s)\in S_gS_{g^{-1}}$ and  $\epsilon_g'(s)\in S_{g^{-1}}S_g$  such that  $\epsilon_g(s)s=s=s\epsilon_g'(s)$.
\end{enumerate}
\end{pro}

We present a Dade like Theorem (see \cite[Theorem 2.8]{D}) for the class of nearly epsilon-strongly graded rings. First we give the following.
\begin{lem}\label{sres}
Let  $M$ in $S$-gr,  and write $S_{(g,g\m)}=S_gS_{g^{-1}}, g\in G,$  then $S(M):=\displaystyle\bigoplus_{g\in G}S_{(g,g\m)}M_g$ is a graded  submodule of $M$. 
\end{lem}
\begin{proof}
Take $h\in G$, we need to show that  $S(M)\cap M_h=S_{(h,h\m)}M_h$. It is clear that $S(M)\cap M_h\supseteq S_{(h,h\m)}M_h,$ for the other inclusion take $m\in S(M)\cap M_h$  a non zero-element and write $m=\sum_{i=1}^n m_{g_i}$ the homogeneous decomposition of $m,$ for some  $n\in \mathbb{N}$ and $m_{g_i}\in S_{(g_i,g\m_i)}M_{g_i}$. If $h\notin\{g_1\cdots g_n\},$ then    $m\in (M_{g_1}+\cdots +M_{g_n})\cap M_h=\{0\},$ which is a contradiction and $h\in\{g_1,\cdots g_n\}$ then $m=m_h \in S_{(h,h\m)}M_h$.
\end{proof}
Because of Lemma \ref{sres} we can consider the covariant  restriction functor $S({\rm id}):S\text{-}gr\rightarrow S\text{-}gr$ sending a graded module $M$ to its submodule $S(M).$

Now we give a version of  Dade's theorem (see \cite[Theorem 2.8]{D}) for  the class of nearly epsilon-strongly graded rings.
\begin{teo}\label{teo1.1}
The following statements are equivalent:\begin{enumerate}
    \item [$(i)$] $S$ is nearly epsilon-strongly graded;
    \item [$(ii)$] For  $M\in S$-gr  and $g,h\in G$, $S_{(g,g^{-1})}$ is s-unital and  $S_gM_h=S_{(g,g^{-1})}M_{gh}$;
    \item [$(iii)$] For $M\in S$-gr and $g,h\in G$, $S_{(g,g^{-1})}$ is s-unital and the multiplication map  $m_{g,h}:S_g\otimes M_h\rightarrow S_{(g,g^{-1})}M_{gh}$ is an isomorphism in R-mod;
    \item [$(iv)$] For $g\in G$ the ring $S_{(g,g^{-1})}$ is s-unital, and the family of multiplications maps 
        $\tau=\{\tau_M:S\otimes M_e\rightarrow \bigoplus_{g\in G}S_{(g,g^{-1})}M_g\}_{M\in S-gr}$ is a natural isomorphism between the functors $S ({\rm id})$ and $S\otimes (-)_e$.
\end{enumerate}
\end{teo}
\begin{proof} $(i)\Rightarrow (ii)$ Take $M\in S$-gr  and $g,h\in G$.  By  (ii) of Proposition\ref{pro1.5} the ring $S$ is symmetrically graded and $S_{(g,g^{-1})}$ is $s$-unital, then $S_gM_h=S_gS_{g^{-1}}S_gM_h\subseteq S_gS_{g^{-1}}M_{gh}\subseteq S_gM_h$, and $S_gM_h=S_{(g,g\m)}M_{gh}.$ For $(ii)\Rightarrow (i),$ it is enough to show that $S$ is symmetrically graded. Take $M=S,$ then $S_g=S_gR=S_gS_{g\m}S_g,$ as desired. Now for $(i)\Rightarrow (iii)$ it is enough to show that  the map $m_{g,h}$ is a monomorphism. For this, take 
 $x=  \sum_{i=1}^n s_g^{(i)}\otimes m_h^{(i)}\in \ker(m_{g,h}),$ since  $S$ is nearly epsilon-strongly graded there exists  $e \in S_gS_{g^{-1}}$ such that $es_{g}^{(i)}=s_g^{(i)},$ for  $1\leq i\leq n$. Write  $e=\sum_{j=1}^{n_p}u_{g}^{(j)}v_{g^{-1}}^{(j)}$, then:
\begin{align*}
   x&=\sum_{i=1}^n s_g^{(i)}\otimes m_h^{(i)}=\sum_{i=1}^n es_g^{(i)}\otimes m_h^{(i)}=\sum_{i=1}^n\sum_{j=1}^{n_p} u_g^{(j)}v_{g^{-1}}^{(j)}s_g^{(i)}\otimes m_h^{(i)}\\
&=\sum_{i=1}^n\sum_{j=1}^{n_p} u_g^{(j)}\otimes v_{g^{-1}}^{(j)}s_g^{(i)} m_h^{(i)}=\sum_{j=1}^{n_p} u_g^{(j)}\otimes v_{g^{-1}}^{(j)}m_{g,h}(x)=0.
\end{align*}
and $m_{g,h}$ is an  isomorphism. Now we show $(iii)\Rightarrow (iv).$  Let $M\in S$-gr, it is clear that $\tau_M$ is a morphism in $S$-gr, and  $(\tau_M)_g=m_{g,e}, g\in G$ which implies that it is surjective.

Otherwise take $g\in G$ y $x\in ker(\tau_M)_g$ with  $x=\sum_{i=1}^na_i\otimes b_i,$  $a_i\in S_g$ y $b_i\in M_e$ for all $i$. Then  $0=\sum_{i=1}^na_ib_i=m_{g,e}(x)$ and $x\in ker(m_{g,e})=\{0\}$ which shows that  $ker(\tau_M)=\{0\}$. Moreover, the naturality of $\tau$ is clear. 
%
  Finally to show $iv)\Rightarrow (i)$ holds, again by (ii) of Proposition\ref{pro1.5} it is enough to show that $S$ is symmetrically graded, but  $S_g=S_gR=\tau_S(S_g\otimes R)=S_gS_{g^{-1}}S_g$, which ends the proof.
\end{proof}

\subsection{On epsilon-strongly graded rings}

Epsilon strongly graded ring were recently introduced and studied in \cite{NYOP}, for the reader's convenience we record  this concept.

\begin{defi}\label{definitionepsilon}
The ring  $S$ is {\it epsilon-strongly graded by} $G$ if it is unital and
 for each $g \in G$ the  $(S_g S_{g^{-1}}, S_{g\m}S_g)$-bimodule $S_g$ is unital. 
\end{defi}

Recall a characterization of epsilon-strongly graded rings.

\begin{pro} \cite[Proposition 7]{NYOP}\label{epsilon1} 
The following assertions are equivalent:
\begin{enumerate}

\item[$(i)$] $S$ is epsilon-strongly graded by $G$; 

\item[$(ii)$] $S$ is symmetrically graded and for every $g \in G$ the $R$-ideal
$S_g S_{g^{-1}}$ is unital;

\item[$(iii)$] For every $g \in G$ there is an element $\epsilon_g \in S_g S_{g^{-1}}$
such that for all $s \in S_g$ the relations $\epsilon_g s = s = s \epsilon_{g^{-1}}$ hold;

\item[$(iv)$] For every $g \in G$ the left $R$-module $S_g$ is finitely generated
and projective and the map $n_g : (S_g)_R \rightarrow {\rm Hom}_R( {}_R S_{g^{-1}} , R )_R$,
defined by $n_g(s)(t) = ts$, for $s \in S_g$ and $t \in S_{g^{-1}}$,
is an isomorphism of $R$-modules.

\end{enumerate}
\end{pro}
The following is clear.
\begin{coro}\label{epcero} Suppose that $S$ is epsilon-strongly graded and let $\ep_g$  be the identity of $S_gS_{g\m}, g\in G.$ Then  $\ep_g=0$ if and only if $S_g=\{0\}.$
\end{coro}
\begin{rem} It is shown in  \cite[Proposition 2.6.7]{nas04} that when $S$ is strongly graded then the functors $\ind$ and $\coind$ are naturally equivalent,  however in  \cite[Remark  2.6.8]{nas04} the authors observe that there are  not strongly graded rings for which $\ind \simeq \coind,$ indeed it is enough to take a non trivial group $G\neq \{e\}$ and a  unital ring $R$ with trivial gradation, that is $R_e=R$ and $R_g=\{0\},$ for all $g\neq e.$ Notice that in this case $R$ is  epsilon-strongly graded with $\epsilon_e=1_R$ and $\epsilon_g=0, g\neq e.$ 
\end{rem}

The next result provides a characterization of epsilon-strongly graded rings in terms of the functors Ind and Coind.

\begin{pro}\label{indcoind}
The following assertions are equivalent.
\begin{enumerate}

\item[$(i)$] $S$  is epsilon-strongly graded;

\item[$(ii)$] For every $g\in G$ the left R-module $S_g$ is finitely generated and projective, and the map 
        $\eta_g':\, _RS_g\rightarrow\, _Rhom_R(_RS_{g^{-1}},\, _RR)$,
 defined by $\eta_g'(s):S_{g^{-1}}\rightarrow R$  $\eta_g'(s)(r)=rs,$ for all $r\in R, s\in S_{g\m}$ is an isomorphism of left R-modules;

\item[$(iii)$] $\ind \simeq \coind.$
\end{enumerate}
\end{pro}
\begin{proof} The proof of  $(i) \Leftrightarrow (ii)$ is similar to $(i) \Leftrightarrow (iv)$ in Proposition \ref{epsilon1} and the equivalence  $(ii) \Leftrightarrow (iii)$ follows from \cite[Theorem 2.6.9]{nas04}.
\end{proof}

We proceed with  the next.

\begin{coro}
Let $S$ be an  epsilon-strongly graded ring and $g\in G$. Then:\begin{enumerate}
    \item [$(i)$] The functor $\ind\circ (-)_g$ preserves the $gr-$simplicity. 
    \item [$(ii)$] If $M\in S\text{-}gr$ is injective and  $\C_g$-torsion free, then  $\ind(M_g)$ is injective.
    \item [$(iii)$]  If $M\in S\text{-}gr$ is projective,  then $M_g$ is projective in $R$-mod.
\end{enumerate}
\end{coro}
\begin{proof}

$(i)$ Let  $M$ be a $gr$-simple module and  $g\in G$. If $M_g=0$ we are done. Suppose  $M_g\neq 0$ and take  $0\neq m\in M_g,$ the $gr$-simplicity of $M$ implies  $\bigoplus_{h\in G}S_{hg^{-1}}m=M,$ and thus $Rm=M_g$,  then $M_g$ is a simple $R$-module.  On the other hand, by the first part of   \cite[ Theorem 2.7.2]{nas04} we have  that $\ind(M_g)/t_{C_e}(\ind(M_g))$ is $gr-$simple. Finally since   $S$ is epsilon-strong then Proposition \ref{indcoind} implies  $\ind(M_g)\simeq \coind(M_g)$ as $R$-modules, hence $t_{\C_e}(\ind(M_g))=0,$ and we conclude that $\ind(M_g)$ is $gr$-simple.

$(ii)$ Take $h\in G.$ Since  $t_{C_g}(M)=0$ we have by \cite[ Proposition 2.6.3]{nas04} that the  morphism $\mu_{g,M}:M \rightarrow \coind(M_g)(g^{-1})$ defined by  $\mu_{g,M}(m):S\ni s\mapsto s_{gh^{-1}}m_h\in M_g$ is a  monomorphism, where $s=\sum_{l\in G}s_l$ is the homogeneous decomposition of $s.$ Now the fact that $M$ is $gr$-injective implies that  $Im(\mu_{g,M})$ is a direct summand of  $\coind(M_g)(g^{-1})$. But by \cite[ Proposition 2.6.2]{nas04} we obtain that  $Im(\mu_{g,M})$ is an essential  $gr$- submodule of  $\coind(g)(g^{-1})$, whence  $Im(\mu_{g,M})=\coind(M_g)(g^{-1})$. Therefore  $M\simeq \coind(M_g)(g^{-1})$ and $\ind(M_g)\simeq \coind(M_g)\simeq M(g)$, which shows that $\ind(M_g)$ es injective.

$(iii)$  Since  $\ind$ is right exact  and  $\coind$ is left exact, the natural isomorphism $Ind\simeq Coind$ implies that  $\ind$ and $\coind$ are both exact.  Then $T_{g^{-1}}\circ \coind$ is exact and by  a. of  \cite[Theorem 2.5.5]{nas04} we have that  $((-)_g, T_{g^{-1}}\circ \coind)$  is an adjoint pair, then the fact that $R$-gr has enough projectives implies that   $(-)_g$ preserves  the projectivity property, from this we conclude that  $M_g$ is projective as $R-$module.
%
\end{proof}

{\begin{rem}
	It must be highlighted that, nearly epsilon-strongly graded rings can be non-unital although epsilon-strongly graded rings have to be unital. Moreover, the readers should keep in mind that our approach to these class of rings is under the assumption that they are unital.
\end{rem}

It is clear that every epsilon-strongly graded ring is nearly epsilon-strongly graded, but not every nearly epsilon-strongly graded ring is strongly graded (see \cite[Example 7.4]{L2}). The following result gives a criteria to determine when a nearly epsilon-strongly graded ring is epsilon-strongly graded 

\begin{teo}\label{neie}
The ring  $S$ is  epsilon-strongly  graded ring if and only if  it is nearly epsilon-strongly graded and  $S_g$ is a finitely generated left  $R$-module, for all $g\in G.$ 

\end{teo}
\begin{proof} By Proposition \ref{epsilon1}  follows  that every epsilon-strongly graded ring $S$  is nearly epsilon-strongly graded and $S_g$ is finitely generated, for all $g\in G.$ For the converse, take $g\in G,$ by the Proposition \ref{indcoind} it is enough to show that   $\eta_g': \,_RS_g\rightarrow \,_Rhom_R(_RS_{g^{-1}},\,_RR),$ defined by $ \eta_g'(s_g)(s'_{g\m}) =s'_{g\m}s_g ,$ where $ s_g\in S_g$ and $ s'_{g\m}\in S_{g\m}$ is an  isomorphism in $R$-mod and that $S_g$ is projective as a left $R$-module.  Take  $s\in \ker \eta_g'$ then  $S_{g^{-1}}s=\{0\}$, and by iii) of Proposition \ref{pro1.5}   we   get $0=\epsilon_g(s)s=s$,  and  $\eta_g'$ is a monomorphism. To prove the surjectivity consider  $s_1,\cdots, s_k\in S_{g\m}$ a generating set and  $\varphi\in hom_R(S_{g^{-1}},R)$.  Since $S_{g\m}$ is a  right $s$-unital $S_gS_{g^{-1}}$-module then by  \cite[Theorem 1]{T} there is $\epsilon'\in S_gS_{g^{-1}}$ such that  $s_i\epsilon'=s_i,$ for  $i=1,\cdots,k$. Write  $\epsilon'=\sum_{i=1}^ka_ib_i, a_i\in S_g, b_i\in S_{g\m}$ and take  $b\in S_{g^{-1}},$ then there are  $r_1,\cdots r_k\in R$ such that $b=\sum_{i=1}^kr_is_i$, whence $b\epsilon'=b$ and 
	$\varphi(b)=\varphi(b\epsilon')=b\sum_{i=1}^ka_i\varphi(b_i)=\eta_g'(\sum_{i=1}^ka_i\varphi(b_i))(b)$,
which gives $\eta_g'(\sum_{i=1}^ka_i\varphi(b_i))=\varphi$. To check that  $S_{g^{-1}}$ is projective define $\varphi_i: S_{g^{-1}}\rightarrow R,$ for $ 1\leq i\leq k$ such that  $s\mapsto sa_i, s\in S_{g^{-1}}$. Then $\varphi_i\in hom_R(S_{g^{-1}},R)$ and for  $s\in S_{g^{-1}}$, one has that 
	$s=s\epsilon'=\sum_{i=1}^n(sa_i)b_i=\sum_{i=1}^n\varphi_i(s)b_i,$ which shows that 
 $S_{g^{-1}}$ is a left  projective $R$-module, and we conclude that $S$ is epsilon-strongly graded. \end{proof}

\subsection{The category of symmetrically graded modules}
\begin{defi}\label{symmodules} Let $M$ be a graded left  $S$-module.
We say that $M$ is left  \emph{symmetrically graded} if the equality  $M_g = S_{(g,g\m)} M_{g}$ 
holds for each $g\in G$.

\end{defi}

We denote by  SIM$S$-gr 
the full subcategory of $S$-gr whose objects are symmetrically $G$-graded  left  $S$-modules.   For $g\in G$ denote $\C'_g=\{M\in{\rm SIM}$S$-gr\mid M_g=0 \}.$ Using this category we obtain  a new characterization of strongly graded rings. Indeed, we have the following.
\begin{teo}\label{ncar} The following assertions are equivalent.
\begin{enumerate}
\item[$(i)$] S is strongly graded;
\item[$(ii)$] S is symmetrically graded and $C'_g=\{0\},$ for all $g\in G.$
\end{enumerate}
\end{teo}
\begin{proof}  The part $(i) \Rightarrow (ii)$ follows from \cite[Proposition 2.6.3]{nas04}  part $g.$ For $(ii) \Rightarrow (i),$ let $M$ in  {\rm SIM}$S$-gr ,  $h\in G$ and  consider the  $S$-graded module $SM_h,$ where $(SM_h)_g=S_{gh\m}M_h,$  then $SM_h$ is a graded submodule of $M,$  and  $\displaystyle\frac{M}{SM_h}$ is an object in {\rm SIM}$S$-gr. Indeed, for $g\in G$ and  $m\in M_g$, there are  $x_i\in S_g S_{g^{-1}}$  and  $y_i\in M_g$ such that $m=\sum x_iy_i$, whence  $\overline{m}=\sum x_i \overline{y_i}\in S_gS_{g^{-1}}\left(\displaystyle\frac{M}{SM_h}\right)_g$.  On the other hand, $\left(\displaystyle\frac{M}{SM_h}\right)_h=0$  which gives $\displaystyle\frac{M}{SM_h}\in C'_h$ and thus $M=SM_h,$ and the fact that $SM_h$ is a graded submodule of $M$ implies $M_{gh}=(SM_h)_{gh}=S_gM_h$  for  $g,h\in G$. In particular, letting $S=M$  we have  $S_gS_h=S_{gh}$ then $S$ is strongly graded.
\end{proof} 

Concerning the category {\rm SIM}$S$-gr  we also have the next.
\begin{pro}\label{catequiv} Suppose that $S$ is  epsilon-strongly graded. Then the  following assertions hold.

\begin{itemize}

\item[$(i)$] The  categories  {\rm SIM}$S$-gr and $R$-{\rm mod} are equivalent via the functors $()_e$ and $\ind.$
\item[$(ii)$] The set $\C'_e$ consists of the zero module.

\item[$(iii)$] A morphism $\phi\colon M\to N$ in {\rm SIM}$S$-{\rm gr} is monic, epic or iso, respectively, if and only if, its restriction $\phi\colon M_e\to N_e$ is monic, epic or iso, respectively in $R$-{\rm mod}.
\end{itemize}
\end{pro}
\begin{proof}(i) Suppose that $S$ is epsilon-strongly graded.  Then the restriction of the functor $S(id)$ to the category  {\rm SIM}$S$-gr is the identity functor. By  Proposition  \ref{teo1.1}, it is enough to show that $S\ot ()_e $ induces an  endofunctor of the category  {\rm SIM}$S$-gr, but  this follows from the fact that $S$ is symmetrically graded. To show (ii), we have by 
  take $M$ in  SIM$S$-gr then there exists  $N$ in $R$-mod  such that $M$ and  $\ind(N)$ are isomorphic as graded modules.  On the other hand, by   (iii) of Proposition \ref{indcoind}  we have   $\ind(N)\simeq \coind(N)$ then $t_{\C_e}(M)=t_{\C_e}( \coind(N))=0,$ and thus $C'_e=0.$
For (iii), since $\ep_e=1_S,$ then using  (i)  the  proof follows exactly as in \cite[Corollary 2.10]{D}. 
\end{proof}

\section{Strongly graded rings inside  epsilon-strongly graded rings} \label{epfort}
Now we study epsilon-strongly graded rings from another perspective, we establish  some conditions to write them as a direct sum of strongly graded rings and a ring with induced trivial gradation.  Some of the ideas presented here are inspired by the works  \cite{KuoSz} and \cite{KuoSzII}.

{\it From now on, in this work  $S$ will be epsilon-strongly graded by $G,$ and $\ep_g$  will be  the identity element of $S_gS_{g\m},$ for every $g\in G,$ in particular we set $\epsilon_e=1_S.$}


\begin{rem}\label{definitiongamma}
Take $g \in G$.
From the relation $\epsilon_g \in S_g S_{g^{-1}}$ 
it follows that there is $n_g \in \mathbb{N}$,
and $u_g^{(i)} \in S_g$ and $v_{g^{-1}}^{(i)} \in S_{g^{-1}}$,
for $i \in \{1,\ldots,n_g\}$,
such that $\sum_{i=1}^{n_g} u_g^{(i)} v_{g^{-1}}^{(i)} = \epsilon_g$.
Unless otherwise stated, the elements 
$u_g^{(i)}$ and $v_{g^{-1}}^{(i)}$ are fixed.
We also assume that $n_e = 1$ and $u_e^{(1)} = v_e^{(1)} = 1$.
Define the additive function 
\begin{equation}\label{defgam}\gamma_g : S\ni s \mapsto \sum_{i=1}^{n_g} u_g^{(i)} s v_{g^{-1}}^{(i)}\in S .\end{equation} 
It is observed in \cite[Remark 14]{NYOP} that the collection of (restriction) maps $ \gamma=\{\gamma_g : Z(R)\epsilon_{g\m} \to Z(R)\epsilon_{g}\}_{g \in G},$ yields a partial action of $G$ on $Z(R)$ (see \cite[Definition 1.1]{DE}) . In particular, for $g,h\in G$  we have by  \cite[(5) p.1939]{DE} that 
\begin{equation}\label{idemps}\gamma_g(\epsilon_{h}\epsilon_{g\m})=\epsilon_{gh}\epsilon_{g}.\end{equation}

Moreover by  \cite[Proposition 13]{NYOP} we know that 
\begin{equation}\label{centra}\gamma_g(r)s_g=s_gr,\end{equation} for all $r\in Z(R)$ and $s_g\in S_g, g\in G.$

\end{rem}

\begin{defi}
Let  $H\subseteq G,$ an element  $r \in Z(R)$ is said to be  $\gamma_H-$invariant, if  $\gamma_h(r\epsilon_{h^{-1}})=r\epsilon_h$ for any  $h\in H$. In particular, if $r$ es $\gamma_G-$invariant, we say that  $r$ is $\gamma$-invariant.
\end{defi}

\begin{Notac}

\begin{enumerate}
    \item [$(i)$] For  $X\subseteq G$  we set $\mathcal{E}_X:=\{\epsilon_x:\ x\in X\}$ and $B(\mathcal{E}_X)$ the boolean semigroup generated  $\mathcal{E}_X$. Further, we let $B(\mathcal{E}_X)^*:=B(\mathcal{E}_X)\setminus\{0\}$. There is a partial order in  $B(\mathcal{E}_X)$ defined by:
    $a\leq b$ if and only if  $a=ab$ for all  $a,b\in B(\mathcal{E}_X)$.

    \item [$(ii)$] For  $r\in B(\mathcal{E}_G)^*$, we set 
$N(r):=\{g\in G: r\epsilon_g=r\}=\{g\in G: r\leq \epsilon_g\}$.  Notice that    $e\in N(r)$ and 
 $r\in B(\mathcal{E}_G)^*$  is minimal if and only if  
\begin{equation}\label{min}N(r)=\{g\in G: \epsilon_gr\neq 0\}.\end{equation}
\end{enumerate}
\end{Notac}

Now we obtain a similar result to \cite[Proposition 5]{KuoSzII} but the proof of our result is different since we are not assuming that the group $G$ is finite.
\begin{pro}\label{lem2.1}
Let $r\in B(\mathcal{E}_{G})^*$ a minimal element. Then the  following assertions are equivalent :
\begin{enumerate}
    \item [$(i)$] $r$ is $\gamma$-invariant.
    \item [$(ii)$] $r\in Z(S)$.
 \item [$(iii)$] $N(r)$ is a subgroup of  $G.$
\end{enumerate}
\end{pro}

\begin{proof} 
We start with  $(i)\Rightarrow (ii).$ Suppose that $r$ is $\gamma$-invariant, it is enough to show that $r$ commutes with any homogeneous element in $S.$ Take $s_g\in S_g, g\in G$ then $rs_g= r\ep_gs_g=\gamma_g(r\ep_{g\m})s_g\stackrel{\eqref{centra}}=s_gr\ep_{g\m}=s_gr$
and $r$ is central.  To prove $(ii)\Rightarrow (i).$ Take $g\in G$ then 
$$\gamma_g(r\ep_{g\m})\stackrel{\eqref{defgam}}=\sum_{i=1}^{n_g} u_g^{(i)} r \ep_{g\m}v_{g^{-1}}^{(i)}=r\gamma_g(\ep_{g\m})=r\ep_g,$$ and $r$ is $\gamma$-invariant. To show $(ii)\Rightarrow (iii).$  Take $g,h\in N(r),$ hence $\ep_gr=\ep_hr=r.$ If $g^{-1}\notin N(r)$ then $0\stackrel{\ref{min}}=r\ep_{g\m},$ which gives $rS_{g\m}=0$ and the fact that $r$ is central implies $rS_gS_{g\m}=0$ and thus $r\ep_g=0,$ which contradicts $g\in N(r).$
Now since   $r$ is $\gamma-$invariant  we have \begin{center}
    $r=r\epsilon_g=\gamma_g(r\epsilon_{g^{-1}})=\gamma_g(r\epsilon_{g^{-1}}\epsilon_h)=r\epsilon_g\epsilon_{gh}=r\epsilon_{gh}$,
\end{center}
and we conclude that $N(r)$ is a subgroup of $G.$ Now suppose that $r$ is minimal. We check   $(iii)\Rightarrow (i),$ for this  take   $g\in G$. If $g\notin N(r)$, then  $g^{-1}\notin N(r)$ and  $\gamma_g(r\epsilon_{g^{-1}})=\gamma_g(0)=0=r\epsilon_g$. Suppose $g\in N(r)$, and take  and  $h_1,h_2,\cdots,h_n\in N(r), n\in\mathbb{N}$ such that  $r=\prod_{i=1}^n\epsilon_{h_i}$.  Since  $g^{-1}\in N(r)$, we have  $g^{-1}h_i\in N(r)$ and  $r=r\epsilon_{g^{-1}h_i}$, for all  $i=1,2,\cdots,n$. 

Then
\begin{align*}
    \gamma_g(r\epsilon_{g^{-1}})&
\stackrel{\eqref{idemps}}=\prod_{i=1}^n\epsilon_{gh_i}\epsilon_g\prod_{i=1}^n\epsilon_{h_i}\epsilon_g=
\prod_{i=1}^n\epsilon_{gh_i}\epsilon_gr\epsilon_g
=\prod_{i=1}^n\epsilon_{gh_i}r\epsilon_g=r\epsilon_g,
\end{align*}
and $r$ is $\gamma$-invariant which gives $(i).$
\end{proof}
%
%
%
%

\begin{defi} Let $r\in B(\mathcal{E}_G)^*$ a minimal element. Then $r$ is called epsilon-central  if satisfies one of the  assertions in Proposition 
\ref{lem2.1}.  
\end{defi}
\begin{lem} \label{lem2.2}
Let  $r, s\in B(\mathcal{E}_G)^*$ be epsilon-central elements. Then
\begin{enumerate}
    \item [$(i)$] $(rS_g)(rS_h)=rS_{gh}$,  for any $g,h\in N(r)$. In particular, $rS=(rS)_{N(r)}=\bigoplus_{g\in N(r)} rS_g$ is strongly $N(r)$-graded.
    \item [$(ii)$] $(rS)_N\oplus (sS)_N$  is strongly $N$-graded, where $N:=N(r)\cap N(s)$.
\end{enumerate}
\end{lem}
\begin{proof}
(i) The fact that $r$ is a central homogeneous idempotent implies that $rS$ is a graded ring with identity $r,$ the minimality of $r$ in $B(\mathcal{E}_G)^*$ gives $r\ep_g=0,$ for all $g\notin N(r)$ and thus  $rS=(rS)_{N(r)}.$ To check that $(rS)_{N(r)}$ is $N(r)$-strongly graded take  $g,h\in N(r)$, then  \begin{center}
     $rS_{gh}=r\epsilon_gS_{gh}\subseteq rS_gS_{g^{-1}}S_{gh}\subseteq rS_gS_h=(rS_g)(rS_h)$.
\end{center}
For  $(ii)$  notice that  $rs=0,$ whence $r+s$ is the identity of  $(rS)_N\oplus (sS)_N.$ Now the fact that it is  $N$-strongly graded is clear. 
\end{proof}

\begin{ejem}\label{ejem2.2}
Let  $A$ be   a commutative ring with a non-zero identity $1_A$ and $B$ a unital ideal of $A$ such that  $1_B\neq 1_A$. Put
\begin{displaymath}
S =  
\left( 
\begin{matrix}
  A & A & B \\
  A & A & B \\
  B & B & A
 \end{matrix}
\right),
\quad
R =  
\left( 
\begin{matrix}
  A & A & 0 \\
  A & A & 0 \\
  0 & 0 & A
 \end{matrix}
\right)
\,\, \text{and} \,\,
T =  
\left( 
\begin{matrix}
  0 & 0 & B \\
  0 & 0 & B \\
  B & B & 0
 \end{matrix}
\right).
\end{displaymath}
Then  by \cite[Proposition 40]{NYOP}  $S$   is an epsilon-strongly   $\mathbb{Z}_2$-graded ring with
$S_0 = R, S_1 = T, \epsilon_{0}=diag(1_A,1_A,1_A)$ and $\ep_{1}=diag(1_B,1_B,1_B)$. Moreover,  $B(\E_G)=\{\ep_{0},\ep_{1}\}$ and $N(\ep_{1})=\mathbb{Z}_2$. Then by   Lemma \ref{lem2.2} one gets that  $\ep_{1}S$ is strongly  $\mathbb{Z}_2$-graded.  Furthermore, $\ep_{1}S=\ep_{1}R\oplus \ep_{1}T$, and 

\begin{displaymath}
\ep_1S=  
\left( 
\begin{matrix}
  B & B & B \\
  B & B & B \\
  B & B & B
 \end{matrix}\right), \quad
\ep_1R =  
\left( 
\begin{matrix}
  B & B & 0 \\
  B & B & 0 \\
  0 & 0 & B
 \end{matrix}
\right)
\,\, \text{ y } \,\,
\ep_1T =  
\left( 
\begin{matrix}
  0 & 0 & B \\
  0 & 0 & B \\
  B & B & 0
 \end{matrix}
\right).
\end{displaymath}
Notice that $\ep_1S$ is  the first known 
example (due to E. Dade, according to \cite[Example 2.9]{dascalescu1999})
of a strongly graded ring which is not a crossed product.
\end{ejem} 
%
%


\begin{pro}\label{pro2.2}
Suppose that  $\{e_1,\cdots, e_k\}\subseteq B(\E_G)^*$  is the set of epsilon-central elements. Then $S=\oplus_{i=1}^k e_iS\oplus e'S$,  where $e'=1_S-\sum_{j=1}^ke_j$. Moreover:\begin{enumerate}
    \item [$(i)$] For   $i=1,2,\cdots,k$ the ring $e_iS$ is strongly $N(e_i)$-graded, and  it is indecomposable into a direct sum of strongly graded rings whose unities belong to $ B(\E_G)^*.$
    \item [$(ii)$] If  $e'=0$, then  $S_N$ is  strongly $N$-graded, where $N=\bigcap_{i=1}^kN(e_i)$.
    \item [$(iii)$] If $e'\neq0$, then  $e'S$  is epsilon-strongly $G$-graded.
\end{enumerate} 
\end{pro}
\begin{proof} To prove (i) we notice that
the fact that   $e_i$ for $1\leq i\leq k$ is epsilon-central, implies that $\{e_1,e_2,\cdots,e_k,e'\}$ is a family of orthogonal  and central idempotents whose sum is $1_S.$  Then $S=\bigoplus_{j=1}^ke_jS\oplus e'S$, moreover the last assertion in (i) follows from the minimality of the $e_i's.$  For (ii)  we have by the first item of  Lemma \ref{lem2.2} that  $e_iS$ is strongly  $N(e_i)$-graded, and if  $e'=0$, then the second item of  Lemma \ref{lem2.2} implies that $S$ is strongly $N$-graded. Now we check (iii), suppose  $e'\neq 0$ 
and let  $\epsilon'_g=e'\epsilon_g, g\in G$, then  
$\epsilon'_g=e'\epsilon_ge'\in e'S_gS_{g^{-1}}e'\subseteq (e'S_g)(e'S_{g^{-1}}),$ and for $s_g\in S_g$ we have that $\epsilon_g'(e's_g)=e'\epsilon_gs_g=e's_g=e's_g\epsilon_{g^{-1}}=e's_ge'\epsilon_{g^{-1}}=(e's_g)\epsilon'_{g^{-1}}, $ and (iii) of Proposition \ref{epsilon1} implies that $e'S$ is epsilon-strongly $G$-graded, which ends the proof.
\end{proof}
%

\begin{ejem}
Let  $S$ be the ring defined in Example \ref{ejem2.2}. Then  $\epsilon_1$ is the only epsilon-central element in $B(\E_{\Z_2})^*$ and by  Proposition \ref{pro2.2} we have  $S=\epsilon_{1}S\oplus (1_S-\epsilon_{1})S$, where  $(1_S-\epsilon_{1})S$ is epsilon-strongly $\mathbb{Z}_2$-graded. But by the proof of Proposition \ref{pro2.2} we obtain   
$\epsilon_{1}'=(1-\epsilon_{1})\epsilon_{1}=0,$ then  $(1_S-\epsilon_{1})S_1=\{0\}$ and we get that $(1_S-\epsilon_{1})S$ has induced trivial gradation.   
\end{ejem}

Adapting \cite[Definition 36]{NyOP3} to the group graded case we have the following.
\begin{defi} We say that a graded  $S$-module $M$ is epsilon-strongly graded if for all $g\in G,$ $S_gS_{g\m}$ is a unital ideal of $S$ and $S_gM_h=S_gS_{g\m}M_{gh}, g,h\in G.$ 
\end{defi}

\begin{coro}Let  $S$  be an  epsilon-strongly  $G$-graded ring. Suppose that  $\{e_1,\cdots, e_k\}\subseteq B(\E_G)^*$  is the set of epsilon-central elements and let $e'=1_S-\sum_{j=1}^ke_j.$ Then for any graded $S$-module $M$ we have $M=\bigoplus\limits_{i=1}^k e_iM\oplus e'M$.  Moreover:
\begin{enumerate}
    \item [$(i)$] For   $i=1,2,\cdots,k$ the set  $(e_iM)_{N(e_i)}$ is  a strongly $N(e_i)$-graded $e_iS$-module  and  it is indecomposable into a direct sum of strongly graded rings whose unities belong to $ B(\E_G)^*.$
    \item [$(ii)$] If  $e'=0$, then  $M_N$ is a  strongly $N$-graded $S_N$-module, where $N=\bigcap_{i=1}^kN(e_i)$.
    \item [$(iii)$] If $e'\neq0$, then  $e'M$  is  an epsilon-strongly $G$-graded $e'S$-module.
\end{enumerate} 
\end{coro}
\begin{proof}Let $M$ be an object in $S$-gr, then the minimality of the idempotents $\{e_1,\cdots, e_k\}$ implies that $M=\bigoplus\limits_{i=1}^k e_iM\oplus e'M.$  Now ({\it i}) and ({\it ii}) follow from (i) and (ii) of Proposition \ref{pro2.2} and \cite[(2.9 b)Theorem 2.8]{D}, finally part ({\it iii}) is a consequence of (iii) in Proposition \ref{pro2.2} and  \cite[Proposition 47]{NyOP3}.
\end{proof}

Let  $S$  be an  epsilon-strongly  $G$-graded ring,  where  $S_g\neq 0,$ for all but a finite subset of $G$. By Corollary \ref{epcero} the set  $B(\E_G)^*$ is finite, suppose that  the set of minimal elements  $\{e_1,\cdots,e_k\}$ are epsilon-central.  If $e'\neq0,$ write  $S^{(1)}:=e'S$ and $1_{(1)}:=e'$. By  Proposition \ref{pro2.2}, the ring  $S^{(1)}$ is epsilon-strongly $G$-graded with $\epsilon_{g}^{(1)}:=e'\epsilon_g=1_{(1)}\epsilon_g, g\in G.$  
We apply to $S^{(1)}$  the same process, let  $\E_{G}^{(1)}:=\{\epsilon_g^{(1)}:g\in G\}.$ First  notice that  $|B(\E_{G}^{(1)})|<|B(\E_G)|$. Indeed, it is clear that the map  $B(\E_G)\ni \prod_{i=1}^n\epsilon_{g_i}\mapsto e'\prod_{i=1}^n\epsilon_{g_i}\in B(\E_G^{(1)})$ is surjective but non injective because for every $i=1,2,\cdots,k$, we have  $e'e_i=0$.

Let  $\gamma^{(1)}$ be the corresponding map of $S^{(1)}$ defined by \eqref{defgam}. Notice that  $r\in S^{(1)}$ is $\gamma^{(1)}$-invariant if and only if is $\gamma$-invariant. 
Consider  $\{e_1^{(1)},\cdots,e_{k_1}^{(1)}\}$ the set of  minimal elements in $B(\E_G^{(1)})$ and suppose they are epsilon-central, since they are orthogonal we can write 
    $S^{(1)}=\oplus_{i=1}^{k_1}e_i^{(1)}S^{(1)}\oplus 1_{(2)}S^{(1)},$
where $1_{(2)}:=1_{(1)}-\sum_{i=1}^{k_1}e_i^{(1)}$ 
and  $e_i^{(1)}S^{(1)}$ is strongly  $N(e_i^{(1)})$-graded for  $i=1,2,\cdots,k_1$.
 If $1_{(2)}=0$, then  $S$ is also strongly $N$-graded where 
   $$N=\bigcap_{i=1}^kN(e_i)\bigcap_{i=1}^{k_1}N(e_i^{(1)}),$$
and the process stops. 
On the other hand, if 
$1_{(2)}\neq 0$, then  $S^{(2)}:= 1_{(2)}S^{(1)}$ is epsilon-strongly graded and we repeat the process.  Again we have $|B(\E_G^{(2)})|<|B(\E_{G}^{(1)})|<|B(\E_G)|,$ and supposing that all minimal elements in  $B(\E_G^{(2)})$ are epsilon-central, we apply the same process to  $S^{(2)}.$ However, since the semigroup $B(\E_G)$ is finite, there exists some integer  $l$ such that $S^{(l)}=1_{(l)}S^{(l-1)}$ is epsilon-strongly  graded, where  for $g\in G$\begin{center}
    $\epsilon_g^{(l)}=1_{(l)}\epsilon_g^{(l-1)}=1_{(l)}1_{(l-1)}\epsilon_g^{(l-2)}=\cdots =1_{(l)}1_{(l-1)}\cdots 1_{(1)}\epsilon_g$.
\end{center}
Moreover,  if the identity of $S^{(l)}$  given by $1_{(l)}=1_{(l-1)}-\sum_{i=1}^{k_{l-1}}e_i^{(l-1)}$, being $\{e_1^{(l-1)},\cdots,e_{k_{l-1}}^{(l-1)}\}$ the set of epsilon-central elements of  $S^{(l-1)},$ is non-zero we must have $\epsilon_g^{(l)}=0,$ for  $g\neq e,$ that is the gradation of $S^{(l)}$ is the trivial one. 

\begin{teo}\label{teo2.1}
Let  $S$  be an  epsilon-strongly  $G$-graded ring,  where  $S_g\neq 0,$ for all but a finite subset of $G$. Following the notations above, if for each step  $j$, $1_{(j)}\neq0$ and the minimal elements in  $B(\E_G^{(j)})$ are epsilon-central, then $S$ is a direct sum of strongly-graded rings and a ring with induced trivial gradation. 
\end{teo}

\subsection{Epsilon-crossed products as sums of crossed products}
Recall that $S$ is an epsilon-crossed product if for any $g\in G$  there is an epsilon-invertible element in $S_g,$ that is an element $s\in S_g$ for which there exists $t\in S_{g^{-1}}$ such that  $st=\epsilon_g$ and $ts=\epsilon_{g^{-1}}$. 
\begin{pro}\label{coro2.1}
Let  $S$  be an  epsilon-strongly  $G$-graded ring satisfying the assumptions of Theorem \ref{teo2.1}.  If $S$  is an epsilon-crossed product then  $S$ is a direct sum of  crossed products and a ring with trivial gradation.
\end{pro}
\begin{proof}
It is enough to show that for any  $e\in B(\E_G)^*$ an epsilon-central element, the ring  $(eS)_{N(e)}=\bigoplus_{g\in N(e)} eS_g$ is a crossed product. Take  $g\in N(e)$ we shall show that   $eS_g\cap U(eS)\neq \emptyset$. First notice that  $e\epsilon_g=e=e\epsilon_{g^{-1}},$ on the other hand there are  $s\in S_g$ and  $t\in S_{g^{-1}}$ such that  $st=\epsilon_g$ y $ts=\epsilon_{g^{-1}}$. Since  $e\in Z(S)$ we have 
    $(es)(et)=est=e\epsilon_g=e=e\epsilon_{g^{-1}}=ets=(et)(es)$.
and this shows $U(eS)\cap eS_g\neq \emptyset,$ as desired.
\end{proof}

\begin{rem} Let $S$ be the ring  defined in Example \ref{ejem2.2}, it is shown in \cite[Proposition 40]{NYOP} that $S$ is not an epsilon-crossed product, and this is done  by means of some straightforward but long calculations.   Now using  Proposition \ref{coro2.1} and  Dade's example, that is  the ring $\ep_1S$ in  Example  \ref{ejem2.2},  we obtain an easier way to show this result.
\end{rem}

We denote by   $G$-gr (resp.  $G$-stgr) the category of $G$-graded (resp. strongly $G$-graded) rings and graded morphisms and  $G$-$\epsilon_G$stgr the  subcategory of  $G$-gr whose objects are the epsilon-strongly $G$-graded rings $S=\oplus_{g\in G}S_g$  for which  $\epsilon_G:=\bigwedge\{\epsilon_g\mid g\in G\}$  belongs to  $B(\E_G)^*.$ It is clear that this element is $\gamma$-invariant and $N(\ep_G)=G.$ 
Notice that the ring in    Example \ref{ejem2.2} is an object in  $G$-$\epsilon_G$stgr.

\begin{pro} \label{fun} The following statements hold:
\begin{itemize}
\item[$(i)$] If $S$ is  an object in $G$-$\epsilon_G$stgr.  Then $\ep_GS=\bigoplus_{g\in G}\ep_GS_g$ is a strongly $G$-graded ring. 
\item [$(ii)$] If $S$ is  is  an object in $G$-$\epsilon_G$stgr and is an epsilon crossed product, then $\ep_GS$ is a crossed product.
\end{itemize}
\end{pro} 
\begin{proof} We have that 
$(i)$ is a consequence of  Proposition \ref{lem2.1} while $(ii)$  follows from Proposition \ref{coro2.1}.
\end{proof}

Our next goal is to show that  $G$-stgr  is reflective in  $G$-$\epsilon_G$stgr. For the reader's convenience we recall this concept here.

\begin{defi} Let $\mathcal{C}$ be a category and   $\mathcal{D}$ be full subcategory of $\mathcal{C}.$ We say that  
 $\mathcal{D}$  is reflective if the inclusion functor $i : \mathcal{D} \to \mathcal{C}$ has a left adjoint. 
\end{defi}

Now we give an auxiliary result.

\begin{pro}\label{carad}\cite[Definition 2.4.3]{P}, \cite[Theorem 2 (i), P 83]{ML}. Let $\mathcal{C}$  and $\mathcal{D}$ be two categories and $F\colon \mathcal{C}\to \mathcal{D},$ $G\colon \mathcal{D}\to \mathcal{C}$ two covariant functors. Then $(F,G)$ is an adjoint pair if there is a natural transformation $\eta\colon 1_{\mathcal{C}}\to GF$ such that for  object $X$ in $\mathcal{C}$ each  object $Y$ in $\mathcal{D}$ and any morphisms $f\colon X\to GY,$ there exists a unique $g\colon FX\to Y$ such that the following diagram commutes.
\[
  \begin{tikzcd}
    X \arrow{r}{\eta_X} \arrow[swap]{dr}{ f} & GFX \arrow{d}{G(g)} \\
     & GY
  \end{tikzcd}
\]
\end{pro}

\begin{pro}\label{refle}  $G$-{\rm stgr}  is a  reflective subcategory of $G$-$\epsilon_G${\rm stgr}. 
\end{pro} 
\begin{proof}  By (i) of Proposition  \ref{fun} there is a functor  $F:G-\epsilon_G{\rm stgr}\to G-{\rm stgr} $ for which $F(S)=\ep_GS$ and $F(f)=f_{\mid \ep_GS}$ for any  morphism $f\colon S\to S'$ in  $G$-$\epsilon_G{\rm stgr}.$ We shall show that $(i, F)$ is an adjoint pair, where $i:  G-{\rm stgr} \to G-\epsilon_G{\rm stgr}$ is the inclusion functor. For this, let $S$ be  an object  in $G$-$\epsilon_G$stgr let $\eta_S\colon S\to \ep_GS$ be the multiplication by $\ep_G.$ Then it is clear that the diagram
\[
\begin{tikzcd}
S \arrow[r,"\eta_{S}"] \arrow[d,swap,"f"] &
  \ep_G M\arrow[d,"F(f)"] \\
S' \arrow[r,"\eta_{S'}"] &  \ep_GS'
\end{tikzcd}
\]
 commutes, and thus $\eta\colon {1}_{G-\epsilon_G{\rm stgr}}\to iF $ is a natural transformation. Now let  $S$ in $G$-$\epsilon_G{\rm stgr}$ and  $S'$ in   $G$-{\rm stgr} and $f\colon S\to S' $ a morphism in  $G$-$\epsilon_G{\rm stgr}$,  then $g\colon \ep_GS\to S'$ defined  by $g=f_{\mid \ep_GS}$ is the only morphism in   such that $G$-{\rm stgr} $\eta_S\circ g=f,$ then by Proposition  \ref{carad} we conclude that $(i, F)$ is an adjoint pair, which ends the proof.
\end{proof}

Let $G$-{\rm crs}  and   $G$-$\epsilon_G${\rm crs}. be the subcategories of  $G$-{\rm stgr}  and f $G$-$\epsilon_G${\rm stgr}, whose objects are crossed products and epsilon-crossed products respectively. Using (ii) of Proposition \ref{fun} and proceeding as in the proof of  Proposition \ref{refle} one can show that  $G$-{\rm crs}  is reflective in  $G$-$\epsilon_G${\rm crs} .

\section{Examples: Leavitt Path Algebras}
The class of Leavitt Path Algebras has become an  important source of examples   in the study of epsilon and  nearly epsilon-strongly graded rings. In this final section we illustrate some of our results using this family.

A directed graph $E=(E^0,E^1,r,s)$ consists of two countable sets $E^0$, $E^1$ and maps $r,s : E^1 \to E^0$. The elements of $E^0$ are called \emph{vertices} and the elements of $E^1$ are called \emph{edges}.  For any vertex $v$ we set $s(v)=r(v)=v.$

\begin{defi} \cite{AAP}
Let $E$ be any directed graph and let $R$ be a commutative ring.
The \emph{Leavitt path $R$-algebra $L_R(E)$ of $E$ with coefficients in $R$} is the
$R$-algebra generated by a set $\{v \mid v\in E^0\}$ of pairwise orthogonal idempotents, together with a set of elements $\{f \mid f\in E^1\} \cup
\{f^* \mid f\in E^1\}$, which satisfy the following  relations:
\begin{enumerate}
\item[$(i)$] $s(f)f=fr(f)=f$, for all $f\in E^1$;
\item[$(ii)$] $r(f)f^*=f^*s(f)=f^*$, for all $f\in E^1$;
\item[$(iii)$] $f^*f'=\delta_{f,f'}r(f)$, for all $f,f'\in E^1$;
\item[$(iv)$] $v=\sum\limits_{ \{ f\in E^1 \mid s(f)=v \} } ff^*$, for every $v\in E^0$ for which $s^{-1}(v)$ is non-empty and finite.
\end{enumerate}
\end{defi}

We have the next. 
\begin{pro}  Assume that  the Leavitt Path Algebra $L_{R}(E)$  is endowed with a standard $G$-gradation  (see  \cite[p. 8]{NYO}), and denote by   $X$ the set of formal expressions $\alpha\beta^*,$ where $\alpha, \beta  \in  W.$ Then the  following assertions are equivalent.
\begin{enumerate}
\item[$(i)$] $L_{R}(E)$ is epsilon-strongly graded;
\item[$(ii)$]  For $g\in G,$ the set $L_{R}(E)_g=span_R\{x\in X\mid \partial(x)=g \}$ is finitely generated as a $R$-module.
\end{enumerate}
\end{pro}
\begin{proof}
	 The part $(i)\Rightarrow(ii)$  is clear.  Conversely, for $(ii)\Rightarrow(i).$ The fact that $L_{R}(E)_g=span_R\{x\in \tilde X\mid \partial(x)=g \}$ is finitely generated as a $R$-module for all $g\in G$  implies that $E^0$ is finite, hence $L_{R}(E)$ is unital with unity $\sum\limits_{v\in E^0}v,$ whence  by  \cite[Theorem 4.2]{NYO}  and   Theorem \ref{neie} we conclude $L_{R}(E)$ is epsilon-strongly graded.
\end{proof}
We illustrate  Proposition  \ref{lem2.1} and  Lemma \ref{lem2.2} with the following.

\begin{ejem}
Let  $E$ be the directed graph:
\begin{center}
    $\xymatrix{
    \bullet_{v_1}\ar[r]^-{f}& \bullet_{v_2}& \bullet_{v_3} \ar[l]_-{g} & \bullet_{v_4}
    }$
\end{center}
we consider the Leavitt Path Algebra $L_{R}(E)$ with coefficients in  a commutative ring $R$  and  define a  standard $\mathbb{Z}_4$-gradation on $L_{R}(E)$ as follows:   $\partial(v_1)=\partial(v_2)=\partial(v_3)=\partial(v_4)=0$, $\partial(f)=\partial(g)=2$. We set  \begin{enumerate}
    \item [$(i)$] $S_{0}=span_R\{v_1,v_2,v_3,v_4,fg^*,gf^*\}$  with $\epsilon_{0}=v_1+v_2+v_3+v_4$.
    \item [$(ii)$] $S_{1}=S_3=\{0\},$  $\epsilon_{1}=\epsilon_{3}=0$.
    \item [$(iii)$] $S_{2}=span_R\{f,f^*,g^*,g\}$ and  $\epsilon_{2}=v_1+v_2+v_3$.

\end{enumerate}
Then $B(\E_{\mathbb{Z}_4})=\{\epsilon_{0},\epsilon_{2},0\}$,  and $\epsilon_{2}$ is the minimum of  $B(\E_{\mathbb{Z}_4})^{\ast}$. Moreover, $N(\epsilon_{2})=\{0, 2\}\simeq \Z_2.$ Then by Proposition  \ref{lem2.1} and  Lemma \ref{lem2.2}, we have that  $(v_1+v_2+v_3)S$ is strongly  $\mathbb{Z}_2$-graded. 
\end{ejem}

We give an example where Theorem \ref{teo2.1} can be  applied.
\begin{ejem}Consider the Leavitt path algebra $S$ associated  to the directed graph  
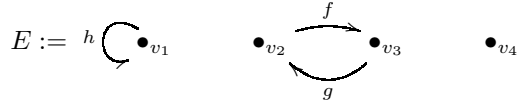
\begin{center}
     $E:=\xymatrix{
    \bullet_{v_1} \ar@(ul,dl)_-{h} & \bullet_{v_2} \ar@/^/[r]^-{f}& \bullet_{v_3} \ar@/^{5mm}/[l] ^-{g}& \bullet_{v_4}
    }$
\end{center}
We give to $S=L_R(E)$ a  standard  $\mathbb{Z}_8$-gradation, where $\partial(f)=\partial(g)=2,$   $deg(h)=4$ and   \begin{enumerate}
    \item [$(i)$] $S_{2}=span_R\{f,g,fgfgf,fgfgfgfgf,\cdots\},$  $\ep_{2}=v_2+v_3.$
    \item [$(ii)$] $S_{4}=span_R\{h,h^*,fg,g^*f^*\cdots\},$  $\ep_{4}=v_1+v_2+v_3.$
    \item [$(iii)$] $S_{6}=span_R\{g^*,f^*,\cdots\},$  $\ep_{6}=v_2+v_3$ and 
\item [$(iv)$] $S_1=S_3=S_5=S_7=\{0\}.$
\end{enumerate}
We have $B(\E_{\mathbb{Z}_8})^*=\{\ep_2, \ep_4\}$ then $ \ep_2$  is minimal and   $N(\ep_2)=\langle 2\rangle\leq \mathbb{Z}_8.$ Then $\ep_2S$ is strongly  $N(\ep_2)$-graded. On the other hand,  $1-\ep_2=v_1+v_4$, and  $S^{(1)}=(v_1+v_4)S$ is epsilon-strongly $\mathbb{Z}_8$-graded. Moreover, $B(\E_{\mathbb{Z}_8}^{(1)})^*=\{v_1+v_4,v_1\} $ and $N(v_1)=\langle 4\rangle.$  Then $v_1S^{(1)}=v_1S$ is strongly $N(v_1)$-graded and we have $$S=(v_2+v_3)S\oplus v_1S \oplus v_4S,$$ where $v_4S$ has trivial gradation. Referring to (i) of Proposition \ref{pro2.2}, notice that $(v_2+v_3)S=v_2S\oplus v_3S$ but neither $v_2$ nor $v_3$ belong to  the boolean semigroup $B(\E_{\mathbb{Z}_8})^*.$
\end{ejem}

\end{document}